\documentclass[10pt,draft,twoside]{amsart}
\usepackage{amssymb}
\usepackage{latexsym}
\usepackage{amsfonts}
\usepackage{amsmath}

\DeclareMathOperator{\Aut}{Aut}

\DeclareMathOperator{\supp}{supp}

\DeclareMathOperator{\z}{{\bf{0}}}
\DeclareMathOperator{\1}{{\bf{1}}}
\DeclareMathOperator{\GL}{GL}

\DeclareMathOperator{\wt}{wt}

\DeclareMathOperator{\base}{\mathfrak{B}}
\DeclareMathOperator{\topg}{\mathfrak{L}}

\newtheorem{theorem}{Theorem}[section]

\newtheorem{lemma}[theorem]{Lemma}

\theoremstyle{definition}

\newtheorem{remark}[theorem]{Remark}

\newtheorem{example}[theorem]{Example}

\newcommand{\F}{\mathbb F}

\renewcommand{\leq}{\leqslant}
\renewcommand{\geq}{\geqslant}

\numberwithin{equation}{section}

\begin{document}

\title[A note on binary completely regular codes with large minimum distance]
      {A note on binary completely regular codes \\ with large minimum distance}
\author{Neil I. Gillespie}
\address{[Gillespie] Centre for the Mathematics of Symmetry and Computation\\
School of Mathematics and Statistics\\
The University of Western Australia\\
35 Stirling Highway, Crawley\\
Western Australia 6009}

\email{neil.gillespie@uwa.edu.au}

\begin{abstract}
We classify all binary error correcting completely regular codes
of length $n$ with minimum distance $\delta>n/2$.  
\end{abstract}

\thanks{{\it Date:} draft typeset \today
{\it Key words and phrases: completely regular codes, Hamming codes, equidistant codes.} 
This research was supported by the Australian Research Council Federation Fellowship FF0776186 of Winthrop Professor Cheryl Praeger.}

\maketitle

\section{Introduction}

We consider codes of length $n$ as subsets of the vertex set $V(\Gamma)=\F_2^n$ of the binary Hamming graph $\Gamma$,
which is endowed with the \emph{Hamming metric} $d(-,-)$.  The graph $\Gamma$ has 
automorphism group $\Aut(\Gamma)=\base\rtimes\topg$, where $\base\cong S_2^n$ 
and $\topg\cong S_n$ \cite[Thm. 9.2.1]{distreg}, and because $\base$ acts regularly 
on $\F_2^n$, we may identify $\base$ with the group of translations of $\F_2^n$ and 
$\topg$ with the group of permutation matrices in $\GL(n,2)$.  We say two codes of length $n$ 
are \emph{equivalent} if there exists $x\in\Aut(\Gamma)$ that maps one to the other.  For
a code $C$ in $\Gamma$, the \emph{minimum distance, $\delta$, of $C$} is the smallest 
distance between distinct codewords.  For $\alpha\in\F_2^n$, the \emph{distance of $\alpha$ from $C$}
is $d(\alpha,C)=\min\{d(\alpha,\beta)\,:\,\beta\in C\}$, and the \emph{covering radius, $\rho$, of $C$} is
the furthest distance any vertex in $\F_2^n$ is from $C$.  We let $C_i$ denote the set of vertices in $\F_2^n$
that are \emph{distance $i$} from $C$.  (For all unexplained concepts, see \cite[Sec. 11.1]{distreg}.)
We say $C$ is \emph{completely regular} if for $\nu\in C_i$, with $i\in\{0,\ldots,\rho\}$, the
number $\ell_{ik}=|\Gamma_k(\nu)\cap C|$ depends only on $i$ and $k$, and not on the choice of $\nu$ (here $\Gamma_k(\nu)$ 
denotes the set of vertices at distance $k$ from $\nu$).  

In his paper on completely regular codes, Neumaier \cite{neu} posed 
the problem of classifying various families of completely regular codes.  With respect to this question, 
we classify all binary completely regular codes of length $n$ with $\delta>\max\{2,n/2\}$.   
An obvious example of one of these codes is the \emph{binary repetition code}, which consists of the
all zero and all one vertices.  Up to equivalence, there exists only one other.

\begin{theorem}\label{main}  Let $C$ be a binary completely regular code with $|C|>1$ and $\delta>\max\{2,n/2\}$.  Then
either $\delta=n$ and $C$ is equivalent to the binary repetition code; or $(n,\delta)=(7,4)$ and
$C$ is equivalent $\mathcal{H}_E$, the even half of the Hamming code given in Example \ref{hamex}.  
\end{theorem}

\begin{remark}  Originally the author believed that Theorem \ref{main} could easily
be deduced from the classification of binary non-antipodal completely regular codes given by Borges et al. \cite{nonantipodal}.
However, recently Borges communicated to the author that there is a mistake in their classification, specifically
stemming from Lemma 14 in their paper.  Furthermore, subsequently Rif{\'a} and Zinoviev \cite{rifa} 
constructed an infinite family of examples that does not appear in their 
classification with Borges (see the codes of length $n=\binom{m}{2}$ for $m$ even given in \cite[Thm. 1(1)]{rifa}).  
This led the author to prove Theorem \ref{main}, and in particular, give a proof that is independent of \cite{nonantipodal}.
Furthermore, this result plays an essential role in the classification of another family of completely regular codes \cite{comtran}.
\end{remark}

\section{Example and Proof}

For $\alpha=(\alpha_1,\ldots,\alpha_n)\in\F_2^n$, the \emph{support of $\alpha$} is the set 
$\supp(\alpha)=\{i\,:\alpha_i\neq 0\}$, and the \emph{weight of $\alpha$} is $\wt(\alpha)=|\supp(\alpha)|$.  We 
denote the unique vertex with $\wt(\alpha)=0$, $n$ by $\z$, $\1$ respectively.  We say a code $C$ 
of length $n$ with minimum distance $\delta$ is a \emph{linear $[n,k,\delta]$-code} 
if it is a $k$-dimensional subspace of $\F_2^n$, and in this case, the \emph{external distance} of $C$
is equal to the the number of non-zero weights of the \emph{dual code of $C$} (see \cite[Sec 11.1]{distreg}).  
We call a set $D$ of vertices of constant weight $k$ in $\F_2^n$ 
a \emph{$t$-design} if $\mathcal{D}=(N,\mathcal{B})$, where $N=\{1,\ldots,n\}$ and 
$\mathcal{B}=\{\supp(\beta)\,:\,\beta\in D\}$, forms a $t-(n,k,\lambda)$ design for some positive integer
$\lambda$.  If $C$ is a binary completely regular code with minimum distance $\delta$
that contains $\z$, then it is known that the set $C(k)$ of codewords of weight $k$, with $\delta\leq k\leq m$, forms
a $t$-design for $t=\lfloor\frac{\delta}{2}\rfloor$, assuming that $C(k)\neq\emptyset$ \cite{upc}.  We now give a 
non-trivial example of a binary completely regular code with $\delta>n/2$.  

\begin{example}\label{hamex}  Let $\mathcal{H}$ be the $[7,4,3]$-Hamming code with the following parity
check matrix: $$H= \left( \begin{array}{ccccccc}
1 & 0 & 0 & 1 & 0 & 1 & 1\\
0 & 1 & 0 & 1 & 1 & 1 & 0\\
0 & 0 & 1 & 0 & 1 & 1 & 1\end{array} \right).$$ 
Let $\mathcal{H}_E$ be the even half of $\mathcal{H}$, so $\mathcal{H}_E$ consists of $\z$ and the set $\mathcal{H}(4)$ of $7$ codewords of
weight $4$.  Interestingly, in this case, $\mathcal{H}_E$ is the dual code of $\mathcal{H}$,  
and is an equidistant code with minimum distance $\delta=4$ \cite[Sec. 3.3]{student}.  Thus,
as $\mathcal{H}$ has weight distribution $(1,0,0,7,7,0,0,1)$, $\mathcal{H}_E$ has external distance $s=3$.  
Consequently, because $\delta=2s-2$ and $\mathcal{H}_E$ consists of
codewords of even weight, it follows that $\mathcal{H}_E$ is completely regular \cite[p.347]{distreg}.  Moreover, we 
deduce that $\mathcal{H}(4)$, which is equal to the set of codewords of weight $4$ in $\mathcal{H}_E$, forms a $2-(7,4,2)$ design.  
\end{example}

For a code $C$ with covering radius $\rho$, the \emph{distance partition of $C$} is the set $\{C,C_1,\ldots,C_\rho\}$, 
which forms a partition of $\F_2^n$.  The distance partition of a code $C$ is \emph{equitable} if, for all $i\geq 0$, every vertex
$x\in C_i$ has the same number $c_i$ of neighbours in $C_{i-1}$ and the same number $b_i$ of neighbours
in $C_{i+1}$.  Neumaier \cite{neu} proved that a code in the Hamming graph (more generally in a distance regular graph) is completely regular 
if and only if its distance partition is equitable.  In this case, $i(C)=\{b_0,\ldots b_{\rho-1},c_1,\ldots,c_\rho\}$
is the \emph{intersection array of $C$}. (By definition, $b_\rho=c_0=0$.) The following result can be found in \cite[Thm 11]{nonantipodal}, 
but we give a new proof here.  

\begin{lemma}\label{useful} Let $C$ be a binary completely regular code with $\delta\geq 3$ such that $\z\in C$ and $\1\notin C$.
Then $C$ has covering radius $\rho\geq\delta-1$ and $C_\rho=\1+C$.  
\end{lemma}

\begin{proof}  As $\1\notin C$, it follows that $\1\in C_i$ for some $i\geq 1$.  Since $C$ is completely regular, 
we deduce that $\Gamma_n(\nu)\cap C\neq\emptyset$ for 
all $\nu\in C_i$.  Hence $\1+C_i\subseteq C$.  Similarly, because the Hamming graph is a distance regular graph, 
we deduce from \cite[Thm 3.2]{neu} that $\Gamma_n(\alpha)\cap C_i\neq\emptyset$
for all $\alpha\in C$, and so $\1+C_i=C$, or equivalently $C_i=\1+C$.  Furthermore, for any $j\in\{0,\ldots,\rho\}$, it
follows that $d(x,C_i)=|i-j|$ for all $x\in C_j$.  Thus, if $i<\rho$ then 
$C_i$ has covering radius $\rho'=\max\{\rho-i, i\}<\rho$, contradicting the fact that $C_i$ is equivalent to $C$.  Hence $i=\rho$.
Now let $\{b_0,\ldots,b_{\rho-1},c_1,\ldots,c_\rho\}$ be the intersection array of $C$.  If $e=\lfloor\delta-1/2\rfloor$, then
$c_i=i$ for $i\leq e$ and $b_i=n-i$ for $i\leq e-1$, and if $\delta$ is even then $b_e=n-e$.  By \cite{neu}, $C_\rho$ 
is completely regular with reverse intersection array.  However, because $C_\rho$ is equivalent to $C$, their intersection arrays   
are in fact equal.  Thus $b_i=c_{\rho-i}$ for $0\leq i\leq \rho-1$.  Now suppose that $\rho<\delta-1$, and so $\rho-e\leq e$.
If $\rho-e<e$ then $n-\rho+e=b_{\rho-e}=c_e=e$, and so $n=\rho<\delta-1$, which is a contradiction.  Thus $\rho=2e$, which
implies that $\delta=2e+2$.  However, in this case $n-e=b_e=c_e=e$, and so $n=\rho<\delta-1$, again a contradiction.
\end{proof}

To prove Theorem \ref{main}, we let $C$ be a binary completely regular code with $|C|>1$ and $\delta>\max\{2,n/2\}$.  
By replacing $C$ with an equivalent code if necessary, we can assume that $\z\in C$.
If $\delta=n$, it is straight forward to deduce that $C=\{\z,\1\}$.  Thus we assume that $\delta<n$.  
Because $C$ is completely regular, it follows that the set $C(\delta)$ of codewords
of weight $\delta$ is non-empty.  Let $\beta\in C(\delta)$.  If $\1\in C$, then $d(\1,\beta)=n-\wt(\beta)<n/2$,
contradicting the minimum distance of $C$.  Thus $\1\notin C$.  Hence, by Lemma \ref{useful}, 
$C$ has covering radius $\rho\geq \delta-1$ and $C_\rho=\1+C$.
Consequently, for $\gamma\in C\backslash\{\z\}$, it holds that 
$$\frac{n}{2}<\delta\leq\wt(\gamma)\leq n-\rho\leq n-\delta+1<\frac{n}{2}+1.$$
In particular, this implies that $n$ is odd, $\delta=(n+1)/2$ and $C=\{\z\}\cup C(\delta)$.  Furthermore, because
$C$ is completely regular, it follows that $C$ is equidistant.  Thus, for all $\alpha,\beta\in C(\delta)$, it holds
that $d(\alpha,\beta)=\delta$.  This implies that $\delta$ is even and that $|\supp(\alpha)\cap\supp(\beta)|=(n+1)/4$
for all $\alpha,\beta\in C(\delta)$.  Consequently there exist positive integers $e,\lambda$
such that $C(\delta)$ forms an $(e+1)-(n,\delta,\lambda)$ design with $\delta=2e+2$ \cite{upc}.  As
$\delta\geq 4$, it follows that $e+1\geq 2$.   Now, a non-negative integer $\ell$ is 
\emph{a block intersection number} of a $t$-design if there exist
two blocks of the design that intersect in exactly $\ell$ points.  We have just shown that the design $C(\delta)$ has
only one block intersection number, which is equal to $(n+1)/4$.  If $e+1\geq 3$, then $C(\delta)$ is at
least a $3$-design, and it is known that the only $3$-designs with one block intersection number are the `degenerate'
cases where $n\in\{\delta, \delta+1\}$ \cite{camnear}, which in this case cannot hold as $4\leq\delta<n/2+1$.  Thus $e+1=2$.
This implies that $\delta=4$ and $m=7$.  Furthermore, because $C(\delta)$ has only one block intersection 
number, it is a symmetric $2-(7,4,\lambda)$ design with $\lambda=(n+1)/4=2$ \cite[Thm 1.15]{camvan}.  Recall from Example
\ref{hamex} the $[7,4,3]$-Hamming code $\mathcal{H}$, and the code $\mathcal{H}_E=\z\cup\,\mathcal{H}(4)$.  We saw in Example
\ref{hamex} that $\mathcal{H}(4)$ forms a $2-(7,4,2)$ design.  The complementary design of this design 
is a $2-(7,3,1)$ design, which is unique up to isomorphism \cite[Table 1.28]{crc}, and so $\mathcal{H}(4)$
is also unique up to isomorphism.  Hence there exists $\sigma\in\topg$ such that $C(\delta)^\sigma=\mathcal{H}(4)$,
and because $\z^\sigma=\z$, it follows that $C^\sigma=\mathcal{H}_E$, proving Theorem \ref{main}.

\end{document}